\documentclass{amsart}
\usepackage{amsmath,amssymb,amsthm}
\newtheorem{thm}{Theorem}[section]

\newtheorem*{conj1}{Extended Lindel\"{o}f Hypothesis}
\newtheorem*{conj2}{Gauss Circle Problem}

\newtheorem{lem}[thm]{Lemma}
\newtheorem*{thm*}{Theorem}
\begin{document}

\title{Visible lattice points and the Extended Lindel\"{o}f Hypothesis}
\author{Wataru Takeda}
\address{Department of Mathematics,Kyoto University, Kitashirakawa Oiwake-cho, Sakyo-ku, Kyoto 606-8502,
Japan}
\email{takeda-w@math.kyoto-u.ac.jp}

\begin{abstract}
We consider the number of visible lattice points under the assumption of the Extended Lindel\"{o}f Hypothesis. We get a relation between visible lattice points and the Extended Lindel\"{o}f Hypothesis. And we also get a relation between visible lattice points over $\mathbf{Q}(\sqrt{-1})$ and the Gauss Circle Problem.
\end{abstract}
\maketitle
\section{Introduction}
Let $K$ be a number field and let $\mathcal{O}_K$ be its ring of integers. We consider an $m$-tuple of ideals $(\mathfrak{a}_1, \mathfrak{a}_2,\ldots,\mathfrak{a}_m)$ of $\mathcal{O}_K$ as a lattice points in $Frac(\mathcal{O}_K)^m$.  When $\mathcal{O}_K=\mathbf{Z}$ they are ordinally lattice points. We say that a lattice point $(\mathfrak{a}_1, \mathfrak{a}_2,\ldots,\mathfrak{a}_m)$ is visible from the origin, if $\mathfrak{a}_1+\cdots+\mathfrak{a}_m=\mathcal{O}_K$. 

There are many results about the number of visible lattice point from 1800's. In the case $K=\mathbf{Q}$, D. N. Lehmer prove that the density of the set of visible lattice points in $\mathbf{Q}^m$ is $1/\zeta(m)$ in 1900 \cite{Le00}. And in general case, B. D. Sittinger proved the number of visible lattice points $(\mathfrak{a}_1, \mathfrak{a}_2,\ldots,\mathfrak{a}_m)$ in $K^m$ with $\mathfrak{Na_i}\le x$ for all $i=1,\ldots,m$ is \[\frac{c^m}{\zeta_K(m)}x^m+\text{(Error term)},\]
where $\zeta_K$ is the Dedekind zeta function over $K$ and $c$ is a positive constant depending only on $K$ \cite{St10}.

Let $V_m(x,K)$ denote the number of visible lattice points $(\mathfrak{a}_1, \mathfrak{a}_2,\ldots,\mathfrak{a}_m)$ with $\mathfrak{Na_i}\le x$ for all $i=1,\ldots,m$. When $K=\mathbf{Q}$, $V_m(x,\mathbf{Q})$ means the number of visible lattice points in $(0,x]^m$. And we let $E_m(x,K)$ denote its error term, i.e. $E_m(x,K)=V_m(x,K)-(cx)^m/\zeta_K(m)$. 

In the case $K=\mathbf{Q}$, we proved that the exact order of $E_m(x,\mathbf{Q})$ is $x^{m-1}$ for $m\ge3$ \cite{Ta16}. But we do not know about the exact order of $E_m(x,K)$. In this paper, we consider better upper order of $E_m(x,K)$ under the situation that the Extended Lindel\"{o}f Hypothesis is true. The statement of our main theorem is the following.
\begin{thm*}
If we assume the Extended Lindel\"{o}f Hypothesis, we get \[E_m(x,K)=O(x^{m-1/2+\varepsilon})\]
for all algebraic number field $K$ and for all $\varepsilon>0$.
\end{thm*}
As a result, we can think of considering the number of visible lattice points as considering the Extended Lindel\"{o}f Hypothesis. And we show that the number of visible lattice points in $\mathbf{Q}(\sqrt{-1})^m$ are associated with the Gauss Circle Problem.

\section{The Extended Lindel\"{o}f Hypothesis}
The Dedekind zeta function $\zeta_K$ over $K$ is considered as a generalization of the Riemann zeta function and $\zeta_K$ is defined as
\[\zeta_K(s)=\sum_{\mathfrak{a}}\frac1{\mathfrak{N}\mathfrak{a}^s},\]with the sum taken over all nonzero ideals of $\mathcal{O}_K$.

Riemann proposed that all non-trivial zeros of the Riemann zeta function is on the line $\Re(s)=1/2$ in his paper \cite{Ri59}. The Extended Riemann Hypothesis over algebraic number field is known as a generalization of the Riemann hypothesis.  The statement of the Extended Riemann Hypothesis is "for all algebraic number field $K$ all non-trivial zeros of the Dedekind zeta function is on the line $\Re(s)=1/2$".

One of other important hypotheses in analytic number theory is the Lindel\"{o}f Hypothesis. As well as the Riemann Hypothesis, this hypothesis can be generalized over algebraic number fields. 
The Extended Lindel\"{o}f Hypothesis can be written as follows. 
\begin{conj1}For $\sigma\ge1/2$
\[\zeta_K(\sigma+it)=O(t^\varepsilon)\]for every $\varepsilon>0$.
\end{conj1}
In 1918, R. Backlund proved that the Lindel\"{o}f Hypothesis is equivalent to the statement that the number of zeros of the Riemann zeta function $\zeta(s)$ in the strip $\{s=\sigma+it~|~1/2<\sigma,\ T\le t\le T+1\}$ is $o(\log T)$ as $T$ tends to $\infty$ \cite{Ba18}. On the other hand, the Riemann Hypothesis stated that all non-trivial zeros of the Riemann zeta function $\zeta(s)$ is on the line $\Re(s)=1/2$, so this hypothesis implies the Lindel\"{o}f Hypothesis.

As well as this result, the Extended Lindel\"{o}f Hypothesis can be followed from the Extended Riemann Hypothesis. Thus the following theorem holds.
\begin{thm}
If the Extended Riemann Hypothesis holds, then for all $\sigma\ge1/2$
\[\zeta_K(\sigma+it)=O(t^\varepsilon)\]for every $\varepsilon>0$.
\end{thm}
From 1900's many results were shown under the situation that the Riemann Hypothesis is true. We assumed the Extended Lindel\"{o}f Hypothesis in this paper, hence we can get same results with assuming the Extended Riemann Hypothesis. We consider a relation between these hypotheses and visible lattice points from the origin in following sections.

\section{Preparation for proof of our main theorem}
In this section, we prepare for showing the main theorem. We consider the number of ideals of $\mathcal{O}_K$ with their ideal norm is less than or equal to $x$. First we need following lemma about complex analysis. 

\begin{lem}
\label{ap}
We have 
\[\frac1{2\pi i}\int_{2-iT}^{2+iT}\frac{x^s}s\ ds=\left\{
\begin{array}{ll}
\displaystyle{O\left(\frac{x^2}{T\log x^{-1}}\right)}& \text{ if } 0<x<1,\\
\displaystyle{\frac12+O\left(\frac1{T}\right)} & \text{ if } x=1,\\
1+\displaystyle{O\left(\frac{x^2}{T\log x}\right)}& \text{ if } x>1. 
\end{array}
\right. \]
\end{lem}
We can prove this lemma by using contour integrals. (For the details for the proof of this result, please see Lemma 4 in section 11 of \cite{Ap76}). We apply this lemma to consider the number of ideals of $\mathcal{O}_K$ with their ideal norm less than or equal to $x$. As shown below, we can compute it with smaller error term by assuming the Extended Lindel\"{o}f Hypothesis.
\begin{thm}
\label{number}
Let $j_K(x)$ be the number of ideals of $\mathcal{O}_K$ with their ideal norm less than or equal to $x$. Assume
the Extended Lindel\"{o}f Hypothesis. Then for every $\varepsilon>0$,
we have\[j_K(x)=c x+O(x^{1/2+\varepsilon}),\]
where \[c=\frac{2^{r_1}(2\pi)^{r_2}hR}{w\sqrt{|d_K|}},\]
and:

$h$ is the class number of $K$,

$r_1$ and $r_2$ is the number of real and complex absolute values of $K$ respectively,

$R$ is the regulator of $K$,

$w$ is the number of roots of unity in $\mathcal{O}^*_K$,

$d_K$ is discriminant of $K$.
\end{thm}
\begin{proof}
It suffices to show that $j_K(x)=c x+O(x^{1/2+\varepsilon})$ for all half integer $x=n+1/2$, where $n$ is a positive integer, because it holds for any real number $y\in[n,n+1)$ that $j_K(x)=j_K(y)$.
\begin{align*}
\intertext{We consider the integral\[\frac1{2\pi i}\int_{2-iT}^{2+iT}\zeta_K(s)\frac{x^s}s\ ds.\]
The series $\displaystyle{\zeta_K(s)=\sum_{\mathfrak{a}}\frac1{\mathfrak{N}\mathfrak{a}^s}}$ is absolutely and uniformly convergent on compact subsets on $\Re(s)>1$. Therefore we can interchange the order of summation and integral in above equation to obtain }
\frac1{2\pi i}\int_{2-iT}^{2+iT}\zeta_K(s)\frac{x^s}s\ ds&=\frac1{2\pi i}\int_{2-iT}^{2+iT}\sum_{\mathfrak{a}}\frac1{\mathfrak{N}\mathfrak{a}^s}\frac{x^s}s\ ds\\
&=\sum_{\mathfrak{a}}\frac1{2\pi i}\int_{2-iT}^{2+iT}\frac1{\mathfrak{N}\mathfrak{a}^s}\frac{x^s}s\ ds.\\
\intertext{Since $x$ was chosen to be a half integer, there are no terms with $\mathfrak{N}\mathfrak{a}=x$ in the above sum. From this, we get}
\frac1{2\pi i}\int_{2-iT}^{2+iT}\zeta_K(s)\frac{x^s}s\ ds&=\sum_{\mathfrak{N}\mathfrak{a}<x}\frac1{2\pi i}\int_{2-iT}^{2+iT}\frac{x^s}{\mathfrak{N}\mathfrak{a}^s}\frac1s\ ds+\sum_{\mathfrak{N}\mathfrak{a}>x}\frac1{2\pi i}\int_{2-iT}^{2+iT}\frac{x^s}{\mathfrak{N}\mathfrak{a}^s}\frac1s\ ds.\\
\intertext{By Lemma \ref{ap} for two sums,}
\frac1{2\pi i}\int_{2-iT}^{2+iT}\zeta_K(s)\frac{x^s}s\ ds&=\sum_{\mathfrak{N}\mathfrak{a}<x}\left(1+O\left(\frac{x^2}{T\mathfrak{N}\mathfrak{a}^2\log {x/\mathfrak{N}\mathfrak{a}}}\right)\right)+O\left(\sum_{\mathfrak{N}\mathfrak{a}>x}\frac{x^2}{T\mathfrak{N}\mathfrak{a}^2\log {\mathfrak{N}\mathfrak{a}}/x}\right)\\
&=j_K(x)+O\left(\frac{x^2}{T}\sum_{\mathfrak{N}\mathfrak{a}<x}\frac1{\mathfrak{N}\mathfrak{a}^2\log {x/\mathfrak{N}\mathfrak{a}}}+\frac{x^2}{T}\sum_{\mathfrak{N}\mathfrak{a}>x}\frac1{\mathfrak{N}\mathfrak{a}^2\log {\mathfrak{N}\mathfrak{a}/x}}\right).
\intertext{Now we estimate how fast the above sums grow. Since $x=n+1/2$}
\left|\sum_{\mathfrak{N}\mathfrak{a}<x}\frac1{\mathfrak{N}\mathfrak{a}^2\log x/\mathfrak{N}\mathfrak{a}}\right|&\le \left(\log {\frac x{x-\frac1{2}}}\right)^{-1}\sum_{\mathfrak{N}\mathfrak{a}<x}\frac1{\mathfrak{N}\mathfrak{a}^2}\\
&=O\left(x\sum_{\mathfrak{N}\mathfrak{a}<x}\frac1{\mathfrak{N}\mathfrak{a}^2}\right),
\intertext{because $\displaystyle{\log\left(1-\frac1{2x}\right)=O(x^{-1})}$ for $x>1$. The sum is bounded by $\zeta_K(2)$, so we estimate}
\left|\sum_{\mathfrak{N}\mathfrak{a}<x}\frac1{\mathfrak{N}\mathfrak{a}^2\log x/\mathfrak{N}\mathfrak{a}}\right|&=O(x).
\intertext{By similar estimate the sum over $\mathfrak{N}\mathfrak{a}>x$ is $O(x)$. Hence}
\frac1{2\pi i}\int_{2-iT}^{2+iT}\zeta_K(s)\frac{x^s}s\ ds&=j_K(x)+O\left(\frac{x^3}{T}\right).
\intertext{We can select large T so that the error term is sufficiently small.}
\intertext{Next we consider the integral  
\[\frac1{2\pi i}\int_{C}\zeta_K(s)\frac{x^s}s\ ds,\]
where $C$ is $C_4C_3C_2C_1$ in the following figure.}
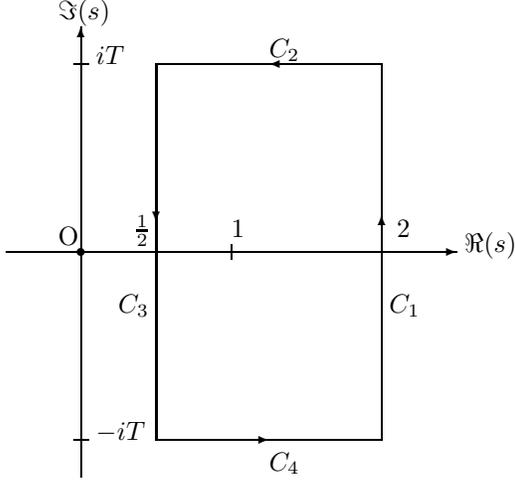
\intertext{\setlength\unitlength{1truecm}
\begin{picture}(6,6)(0,0)
\put(-1,3){\vector(1,0){6}}
\put(0,0){\vector(0,1){6}}
\put(1,0.5){\vector(1,0){1.5}}
\put(2.5,0.5){\line(1,0){1.5}}
\put(4,0.5){\vector(0,1){3}}
\put(4,3.5){\line(0,1){2}}
\put(4,5.5){\vector(-1,0){1.5}}
\put(-0.3,3.1){O}
\put(5.1,3){$\Re(s)$}
\put(-0.3,6.1){$\Im(s)$}
\put(0,3){\circle*{0.1}}
\put(2.5,5.5){\line(-1,0){1.5}}
\put(1,5.5){\vector(0,-1){2.1}}
\put(1,3.5){\line(0,-1){3}}
\put(1,2.9){\line(0,1){0.2}}
\put(-0.1,5.5){\line(1,0){0.2}}
\put(0.2,5.5){$iT$}
\put(-0.1,0.5){\line(1,0){0.2}}
\put(0.2,0.5){$-iT$}
\put(0.7,3.2){$\frac12$}
\put(2,2.9){\line(0,1){0.2}}
\put(2,3.2){$1$}
\put(4.2,3.2){$2$}
\put(2.5,5.6){$C_2$}
\put(4.1,2.2){$C_1$}
\put(2.5,0.1){$C_4$}
\put(0.5,2.2){$C_3$}
\end{picture}}
\intertext{The integral on $C_1$ is the integral which we considered, so we will estimate other integrals.}
\intertext{First we consider the integral over $C_2$ as }
\left|\frac1{2\pi i}\int_{C_2}\zeta_K(s)\frac{x^s}s\ ds\right|&=\left|\frac1{2\pi i}\int_{2}^{1/2}\zeta_K(\sigma+iT)\frac{x^{\sigma+iT}}{\sigma+iT}\ d\sigma\right|\\
&\le\frac1{2\pi}\int^{2}_{1/2}|\zeta_K(\sigma+iT)|\frac{x^{\sigma}}{|\sigma+iT|}\ d\sigma.\\
\intertext{We assume the Extended Lindel\"{o}f Hypothesis, so we have $\zeta_K(\sigma+iT)=O(T^\varepsilon)$,}
\left|\frac1{2\pi i}\int_{C_2}\zeta_K(s)\frac{x^s}s\ ds\right|&=O\left(\int^{2}_{1/2}T^{\varepsilon}\frac {x^{\sigma}}{|\sigma+iT|}\ d\sigma\right)\\
&=O\left(\frac{x^2}{T^{1-\varepsilon}}\right).
\intertext{As well as before, if we select sufficiently large T, the error term will be very small.}
\intertext{Next we calculate the integral over $C_3$ as}
\left|\frac1{2\pi i}\int_{C_3}\zeta_K(s)\frac{x^s}s\ ds\right|&=\left|\frac1{2\pi i}\int_{T}^{-T}\zeta_K\left(\frac12+it\right)\frac{x^{1/2+it}}{\frac12+it}i\ dt\right|\\
&\le\frac1{2\pi}\int^{T}_{-T}\left|\zeta_K\left(\frac12+it\right)\right|\frac{x^{1/2}}{\left|\frac12+it\right|}\ dt.\\
\intertext{We assume the Extended Lindel\"{o}f Hypothesis, so we have $\zeta_K\left(\frac12+it\right)=O(t^\varepsilon)$,}
\left|\frac1{2\pi i}\int_{C_3}\zeta_K(s)\frac{x^s}s\ ds\right|&=O\left(\int^{T}_{-T}T^{\varepsilon}\frac {x^{1/2}}{\left|\frac12+it\right|}\ dt\right)\\
&=O(x^{1/2}T^{\varepsilon}\log T).
\intertext{If we select $T=x^3$, then the order of this error term is $O(x^{1/2+\varepsilon})$.}
\intertext{Finally we estimate the integral over $C_4$ as}
\left|\frac1{2\pi i}\int_{C_4}\zeta_K(s)\frac{x^s}s\ ds\right|&=\left|\frac1{2\pi i}\int^{2}_{1/2}\zeta_K(\sigma-iT)\frac{x^{\sigma-iT}}{\sigma-iT}\ d\sigma\right|\\
&\le\frac1{2\pi}\int^{2}_{1/2}|\zeta_K(\sigma-iT)|\frac{x^{\sigma}}{|\sigma-iT|}\ d\sigma.\\
\intertext{We assume the Extended Lindel\"{o}f Hypothesis, so we have $\zeta_K(\sigma-iT)=O(T^\varepsilon)$,}
\left|\frac1{2\pi i}\int_{C_4}\zeta_K(s)\frac{x^s}s\ ds\right|&=O\left(\int^{2}_{1/2}T^{\varepsilon}\frac {x^{\sigma}}{|\sigma-iT|}\ d\sigma\right)\\
&=O\left(\frac{x^2}{T^{1-\varepsilon}}\right).
\intertext{As well as before, if we select sufficiently large T, we can make the error term very small.}
\end{align*}
By the Cauchy residue theorem we get \[\frac1{2\pi i}\int_{C}\zeta_K(s)\frac{x^s}s\ ds=\rho x,\]where $\rho$ is the residue of $\zeta_K(s)$ at $s=1$. But it is known that $\rho=c$. (For the proof of this result, please see Theorem 5 in Section 8 of  \cite{La90}.) 
\begin{align*}
\intertext{By using all result above, we reach}
j_K(x)+O\left(\frac{x^3}{T}\right)&=cx+O\left(\frac{x^2}{T^{1-\varepsilon}}\right)+O(x^{1/2}T^{\varepsilon}\log T)+O\left(\frac{x^2}{T^{1-\varepsilon}}\right).\\
\intertext{When we select $T=x^3$, this becomes}
j_K(x)&=cx+O(x^{1/2+\varepsilon}).
\intertext{This proves the theorem.}
\end{align*}
\end{proof}
We estimated $j_K(x)$ with assuming the Extended Lindel\"{o}f Hypothesis in Theorem \ref{number}. In Lemma \ref{lem}, we consider the sum $\displaystyle{\sum_{\mathfrak{N}\mathfrak{a}\le x}\mu(\mathfrak{a})j_K\left(\frac x{\mathfrak{N}\mathfrak{a}}\right)^m}$, where $\mu(\mathfrak{a})$ is the M\"{o}bius function defined as 
\[\mu(\mathfrak{a})\overset{def}{=}\left\{
\begin{array}{ll}
1 &i\!f \ \mathfrak{a}=1,\\
(-1)^s &i\!f \ \mathfrak{a}=\mathfrak{p}_1\cdots \mathfrak{p}_s, \text{ where $\mathfrak{p}_1,\ldots, \mathfrak{p}_s$ are distinct prime ideals,}\\
0  &i\!f \  \mathfrak{a}\subset\mathfrak{p}^2 \text{ for some prime ideal } \mathfrak{p}.
\end{array}
\right.\]
As we show in next section, this sum has a crucial role in computing visible lattice points.
\begin{lem}
\label{lem}
If we assume the Extended Lindel\"{o}f Hypothesis, we get
\[\sum_{\mathfrak{N}\mathfrak{a}\le x}\mu(\mathfrak{a})j_K\left(\frac x{\mathfrak{N}\mathfrak{a}}\right)^m=\frac{c^m}{\zeta_K(m)}x^m+O(x^{m-1/2+\varepsilon})\]
for all $\varepsilon>0$.
\end{lem}
\begin{proof}We can show this lemma from last Theorem \ref{number}.
\begin{align*}
\intertext{Theorem \ref{number} and the binomial theorem lead to}
\sum_{\mathfrak{N}\mathfrak{a}\le x}\mu(\mathfrak{a})j_K\left(\frac x{\mathfrak{N}\mathfrak{a}}\right)^m&=\sum_{\mathfrak{N}\mathfrak{a}\le x}\mu(\mathfrak{a})\left(\frac {cx}{\mathfrak{N}\mathfrak{a}}+O\left(\left(\frac x{\mathfrak{N}\mathfrak{a}}\right)^{1/2+\varepsilon}\right)\right)^m\\
&=(cx)^m\sum_{\mathfrak{N}\mathfrak{a}\le x}\frac {\mu(\mathfrak{a})}{\mathfrak{N}\mathfrak{a}^m}+O\left(\sum_{\mathfrak{N}\mathfrak{a}\le x}\left(\frac {x}{\mathfrak{N}\mathfrak{a}}\right)^{m-1/2+\varepsilon}\right).\\
\intertext{By the fact $\displaystyle{\sum_{\mathfrak{a}}\frac {\mu(\mathfrak{a})}{\mathfrak{N}\mathfrak{a}^m}=\frac1{\zeta_K(m)}}$, we get}
\sum_{\mathfrak{N}\mathfrak{a}\le x}\mu(\mathfrak{a})j_K\left(\frac x{\mathfrak{N}\mathfrak{a}}\right)^m&=\frac{c^m}{\zeta_K(m)}x^m-(cx)^m\sum_{\mathfrak{N}\mathfrak{a}> x}\frac {\mu(\mathfrak{a})}{\mathfrak{N}\mathfrak{a}^m}+O\left(\sum_{\mathfrak{N}\mathfrak{a}\le x}\left(\frac {x}{\mathfrak{N}\mathfrak{a}}\right)^{m-1/2+\varepsilon}\right).\\
\intertext{Now we estimate how fast above first sum grows. From Theorem \ref{number} we can estimate $j_K(x)-j_K(x-1)=O(x^{1/2+\varepsilon})$, so we have}
\sum_{\mathfrak{N}\mathfrak{a}> x}\frac {\mu(\mathfrak{a})}{\mathfrak{N}\mathfrak{a}^m}&=O\left(x^m\int_x^{\infty}\frac{y^{1/2+\varepsilon}}{y^m}\ dy\right)\\
&=O(x^{3/2+\varepsilon}).\\
\intertext{Next we estimate how fast above second sum grows. As well as first sum, $j_K(x)-j_K(x-1)=O(x^{1/2+\varepsilon})$ holds, so we have}
\sum_{\mathfrak{N}\mathfrak{a}\le x}\left(\frac {x}{\mathfrak{N}\mathfrak{a}}\right)^{m-1/2+\varepsilon}&=O\left(x^{m-1/2+\varepsilon}\left(1+\int_1^x\frac{y^{1/2+\varepsilon}}{y^{m-1/2+\varepsilon}}\ dy\right)\right)\\
&=\left\{
\begin{array}{ll}
O(x^{m-1/2+\varepsilon}) &i\!f \ m\ge3,\\
O(x^{3/2+\varepsilon}\log x) &i\!f \ m=2.
\end{array}
\right.\\
\intertext{Hence we get final estimate.}
\sum_{\mathfrak{N}\mathfrak{a}\le x}\mu(\mathfrak{a})j_K\left(\frac x{\mathfrak{N}\mathfrak{a}}\right)^m&=\frac{c^m}{\zeta_K(m)}x^m+O(x^{m-1/2+\varepsilon}).
\intertext{This proves this lemma.}
\end{align*}
\end{proof}

\section{The proof of the main theorem}
When we assume the Extended Lindel\"{o}f Hypothesis, we can improve order of error term of some estimates in the last section. In this section we will show the main theorem by using results shown.
\begin{thm}
\label{main}
If we assume the Extended Lindel\"{o}f Hypothesis, we get \[E_m(x,K)=O(x^{m-1/2+\varepsilon})\]
for all algebraic number field $K$ and for all $\varepsilon>0$.
\end{thm}
\begin{proof}
From Theorem \ref{number}, we know that \[j_K(x)=c x+O(x^{1/2+\varepsilon}).\] We use this approximation formula to consider the error term $E_m(x,K)$ by following B. D. Sittinger way \cite{St10}.
\begin{align*}
\intertext{The Inclusion-–Exclusion Principle shows that}
V_m(x,K)&=j_K(x)^m-\sum_{\mathfrak{p}_1}j_K\left(\frac x{\mathfrak{N}\mathfrak{p}_1}\right)^m+\sum_{\mathfrak{p}_1,\mathfrak{p}_2}j_K\left(\frac x{\mathfrak{N}\mathfrak{p}_1\mathfrak{p}_2}\right)^m-\cdots\\
\intertext{where $\mathfrak{p}_1,\mathfrak{p}_2,\ldots$ denote distinct prime ideals with $\mathfrak{N}\mathfrak{p}_i\le x$. By the definition of M\"{o}bius function $\mu(\mathfrak{a})$ we can rewrite this sum simpler}
V_m(x,K)&=\sum_{\mathfrak{Na}\le x}\mu(\mathfrak{a})j_K\left(\frac x{\mathfrak{N}\mathfrak{a}}\right)^m.\\
\intertext{Using the Lemma \ref{lem}, we get}
V_m(x,K)&=\frac{c^m}{\zeta_K(m)}x^m+O(x^{m-1/2+\varepsilon}).
\intertext{This proves the main theorem.}
\end{align*}
\end{proof}

In 2010, B. D. Sittinger showed following theorem about visible lattice points over algebraic number field $K$ without assuming the Extended Lindel\"{o}f Hypothesis.
\begin{thm}
When $n=[K:\mathbf{Q}]$

\[V_m(x,K)=\frac{c^m}{\zeta_K(m)}x^m+\left\{
\begin{array}{ll}
O(x^{m-1/n}) & \text{ if } m\ge3,\\
O(x^{2-1/n}\log x)& \text{ if } m=2,
\end{array}
\right.\]
where $c$ is same constant as before.
\end{thm}
Considering Sittinger's result \cite{St10}, we can improve the order of $E_m(x,K)$ for all algebraic number field $K$ with $[K:\mathbf{Q}]\ge3$ under the situation that the Extended Lindel\"{o}f Hypothesis is true. 

\section{A relation with Gauss Circle Problem}
From Theorem \ref{main} in the last section, we can consider the relation between the number of visible lattice points from the origin over $K$ and the Extended Lindel\"{o}f Hypothesis. In this section, we consider about some relation between the number of visible lattice points from the origin over $\mathbf{Q}(\sqrt{-1})$ and the Gauss Circle Problem. It is well known that the Gauss Circle Problem can be written as follows.
\begin{conj2}
Let $N(r)$ be the cardinality of the set $\{(x,y)\in\mathbf{Z}^2~|~x^2+y^2\le r\}$.
Then \[N(r)=\pi r+O(r^{1/4+\varepsilon}),\]
for every $\varepsilon>0$. 
\end{conj2}
The original statement of the Gauss Circle Problem is not $r$ but $r^2$, but we change its statement slightly to consider a relation with the number of visible lattice points from the origin over $K=\mathbf{Q}(\sqrt{-1})$. Because $\mathcal{O}_K=\mathbf{Z}[\sqrt{-1}]$ is PID, all ideals of $\mathbf{Z}[\sqrt{-1}]$ can be written $(x+y\sqrt{-1})$ and their ideal norm is $x^2+y^2$. Considering the number of units in $\mathbf{Z}[\sqrt{-1}]$, $N(r)=4j_K(r)$ holds.
We can show the following theorem in a way similar to the proof of Theorem \ref{main}. As we remark later, we can assume that $N(r)-\pi r=O(r^{\alpha})$, where $1/4<\alpha<1/3$.
\begin{thm}
\label{circle}
When $K=\mathbf{Q}(\sqrt{-1})$ and $1/4<\alpha<1/3$,
\[N(r)=\pi r+O(r^{\alpha})\] is equivalent to for all $m\ge2$ \[E_m(x,K)=\left\{
\begin{array}{ll}
O(x^{m-1+\alpha}) & \text{ if } m\ge3,\\
O(x^{1+\alpha}\log x)& \text{ if } m=2.
\end{array}
\right.\]
\end{thm}
\begin{proof}
If we assume $N(r)=\pi r+O(r^{\alpha})$ we can get statement about $E_m(x,K)$ in this theorem  by same argument of Theorem \ref{main} and the proof of Lemma \ref{lem} with $1/2+\varepsilon$ replaced by $\alpha$. 

Conversely, if $E_m(x,K)$ satisfies the estimate in this theorem and $N(r)-\pi r\not=O(x^{\alpha})$ then we lead a contradiction as follows. By using same argument of the proof of Lemma \ref{lem}, we get \[\sum_{\mathfrak{N}\mathfrak{a}\le x}\mu(\mathfrak{a})j_K\left(\frac x{\mathfrak{N}\mathfrak{a}}\right)^m-\frac{c^m}{\zeta_K(m)}x^m\not=\left\{
\begin{array}{ll}
O(x^{m-1+\alpha})&i\!f \ m\ge3,\\
O(x^{1+\alpha}\log x)&i\!f \ m=2,
\end{array}
\right.\]
since we can calculate that $c=\pi/4$ and $N(r)-\pi r=O(r^{1/3})$ is well known result about the Gauss Circle Problem. We apply this result to estimate of $V_m(x,K)$, then we get \[V_m(x,K)-\frac{c^m}{\zeta_K(m)}x^m\not=\left\{
\begin{array}{ll}
O(x^{m-1+\alpha}) &i\!f \ m\ge3,\\
O(x^{1+\alpha}\log x) &i\!f \ m=2,
\end{array}
\right.\]
we have a contradiction. Hence we show \[N(r)=\pi r+O(r^{\alpha}).\]
This proves Theorem 5.1. 
\end{proof}
This theorem means that the better order of $E_m(x,\mathbf{Q}(\sqrt{-1}))$ we get, the better order of $N(r)-\pi r$ we get. In 1915, E. Landau and G. H. Hardy proved that \[N(r)-\pi r\not=O(r^{1/4})\] independently \cite{La15} and \cite{Ha15}. 
And the best upper order in now is $0.3149\ldots$ proved by M. N. Huxley \cite{Hu00}, so we know the exact order of $E_m(x,\mathbf{Q}(\sqrt{-1}))$ is less than $0.3149\ldots$ and greater than $1/4$.

\section{Appendix}
In this section we consider some further result about relatively $s$-prime lattice point. They may be a generalization of our results, but we used only visible lattice points to consider some relations with two hypotheses in this paper. Not to mention, there are some relation between relatively $s$-prime lattice point and two hypotheses.

We say that a lattice point $(\mathfrak{a}_1, \mathfrak{a}_2,\ldots,\mathfrak{a}_m)$ is relatively $s$-prime, if there exists no prime ideal $\mathfrak{p}$ such that $\mathfrak{a}_1, \mathfrak{a}_2,\ldots,\mathfrak{a}_m\subset \mathfrak{p}^s$. If $s=1$ then relatively $1$-prime lattice point is visible lattice point from the origin. Follow the definition of $V_m(x,K)$ and $E_m(x,K)$, let $V_m^s(x,K)$ denote the number of relatively $s$-prime lattice point $(\mathfrak{a}_1, \mathfrak{a}_2,\ldots,\mathfrak{a}_m)$ with $\mathfrak{Na_i}\le x$ for all $i=1,\ldots,m$.  And we let $E_m^s(x,K)$ denote its error term, i.e. $E_m^s(x,K)=V_m^s(x,K)-(cx)^m/\zeta_K(ms)$. 

A relation with the Extended Lindel\"{o}f Hypothesis as follows.
\begin{thm}
If we assume the Extended Lindel\"{o}f Hypothesis, we get \[E_m^s(x,K)=\left\{
\begin{array}{ll}
O(x^{3/4+\varepsilon})&\text{ if } m=1 \text{ and } s=2,\\
O(x^{m-1/2+\varepsilon})&\text{ otherwise } ,
\end{array}
\right.\]
for all algebraic number field $K$ and for all $\varepsilon>0$.
\end{thm}
And a relation with the Gauss Circle Problem as follows. 
\begin{thm}When $K=\mathbf{Q}(\sqrt{-1})$ and $1/4<\alpha<1/3$,
\[N(r)=\pi r+O(r^{\alpha})\] is equivalent to for all $ms\ge2$ \[E_m^s(x,K)=\left\{
\begin{array}{ll}
O(x^{m-1+\alpha})&\text{ if } m\ge3, \text{ or } m=2 \text{ and } s\ge2,\\
O(x^{1+\alpha}\log x)&\text{ if } m=2 \text{ and } s=1,\\
O(x^{(1+\alpha)/r})&\text{ if } m=1 \text{ and } s=2,3,4,\\
O(x^{\alpha})&\text{ if } m=1 \text{ and } s\ge5.
\end{array}
\right.\]
\end{thm}
These theorems can be shown to change the proof of theorem about visible lattice points slightly. In a way similar to the proof of Theorem \ref{main}, we get a following estimate of $V_m^s(x,K)$.
\[V_m^s(x,K)=\sum_{\mathfrak{Na}\le \sqrt[s]{x}}\mu(\mathfrak{a})j_K\left(\frac x{\mathfrak{N}\mathfrak{a}^s}\right)^m.\]
All we have to do is considering the order of sum, but we can estimate above sum in a way similar to the proof of Lemma \ref{lem}. Therefore we leave out the last part of proof in this paper.

\end{document}